\documentclass{article}
\usepackage{kpfonts}
\usepackage[T1]{fontenc}
\usepackage{hyperref}
\usepackage{amsmath,amssymb,amsthm,dsfont,url}
\hypersetup{
    colorlinks,
    citecolor=black,
    filecolor=black,
    linkcolor=black,
    urlcolor=black
}

\usepackage{fancyhdr} 
\numberwithin{equation}{section} 

\newtheorem{theorem}{Theorem}[section]
\newtheorem{corollary}[theorem]{Corollary}

\newtheorem{proposition}[theorem]{Proposition}
 \theoremstyle{definition}
 \theoremstyle{remark}
\theoremstyle{question}
 
\newcommand{\comment}[1]{}

\newcommand{\R}{\mathbb R}
\newcommand{\N}{\mathbb N}

\newcommand{\EE}{\mathbb{E}}

\newcommand{\eps}{\varepsilon}

\newcommand{\ls}{\leqslant}
\newcommand{\gr}{\geqslant}

\providecommand{\Prob}[1]{\mathbb{P}\left(#1\right)}

\providecommand{\abs}[1]{\lvert#1\rvert}
\providecommand{\norm}[1]{\lVert#1\rVert}

\providecommand{\conv}[1]{\mathop{\rm conv}\left\{#1\right\}}
\providecommand{\vol}[1]{\left\lvert#1\right\rvert}

\providecommand{\dimm}[1]{\mathop{\rm dim}\left(#1\right)}
\providecommand{\det}[1]{\mathop{\rm det}(#1)}
\providecommand{\var}[1]{\mathop{\rm var}(#1)}

\begin{document}
\large

\title{A central limit theorem for projections of the cube}

\author{Grigoris Paouris \thanks{The first-named author is supported
    by the A. Sloan Foundation, BSF grant 2010288 and the US National
    Science Foundation, grants DMS-0906150 and CAREER-1151711.} \and
  Peter Pivovarov \thanks{The second-named author was supported by a
    Postdoctoral Fellowship award from the Natural Sciences and
    Engineering Research Council of Canada and the Department of
    Mathematics at Texas A\&M University.} \and J. Zinn \thanks{The
    third-named author was partially supported by NSF grant
    DMS-1208962.}}

\date{December 2, 2012}
\maketitle
\begin{abstract}
We prove a central limit theorem for the volume of projections of the
cube $[-1,1]^N$ onto a random subspace of dimension $n$, when $n$ is
fixed and $N\rightarrow \infty$. Randomness in this case is with
respect to the Haar measure on the Grassmannian manifold.
\end{abstract}

\section{Main result}

The focus of this paper is the volume of random projections of the
cube $B_{\infty}^N=[-1,1]^N$ in $\R^N$. To fix the notation, let $n\gr
1$ be an integer and for $N\gr n$, let $G_{N,n}$ denote the
Grassmannian manifold of all $n$-dimensional linear subspaces of
$\R^N$. Equip $G_{N,n}$ with the Haar probability measure $\nu_{N,n}$,
which is invariant under the action of the orthogonal group.  Suppose
that $(E(N))_{N\gr n}$ is a sequence of random subspaces
with $E(N)$ distributed according to $\nu_{N,n}$. We consider the
random variables
\begin{equation}
Z_N = \abs{P_{E(N)} B_{\infty}^N}, 
\end{equation} 
where $P_{E(N)}$ denotes the orthogonal projection onto $E(N)$ and
$\abs{\cdot}$ is $n$-dimensional volume, when $n$ is fixed and
$N\rightarrow \infty$.  We show that $Z_N$ satisfies the following
central limit theorem.

\begin{theorem} 
  \label{thm:CLTcube}
  \begin{equation}
    \label{eqn:CLTcube}
    \frac{Z_N-\EE Z_N}{\sqrt{\var{Z_N}}}\overset{d}{\rightarrow} 
    \mathcal{N}(0,1) \text{ as } N\rightarrow \infty.
  \end{equation}
\end{theorem}
Here $\overset{d}{\rightarrow}$ denotes convergence in distribution
and $\mathcal{N}(0,1)$ a standard Gaussian random variable with mean
$0$ and variance $1$. Our choice of scaling for the cube is immaterial
as the quantity in (\ref{eqn:CLTcube}) is invariant under scaling and
translation of $[-1,1]^N$.

Gaussian random matrices play a central role in the proof of Theorem
\ref{thm:CLTcube}, as is often the case with results about random
projections onto subspaces $E\in G_{N,n}$. Specifically, we let $G$ be
an $n\times N$ random matrix with independent columns $g_1,\ldots,g_N$
distributed according to standard Gaussian measure $\gamma_n$ on
$\R^n$, i.e.,
\begin{equation*}
  d\gamma_n(x) = (2\pi)^{-n/2}e^{-\norm{x}_2^2/2}dx.
\end{equation*}
We view $G$ as a linear operator from $\R^N$ to $\R^n$.  If $C\subset
\R^N$ is any convex body, then
\begin{equation}
  \label{eqn:splitting}
  \abs{GC} = \det{(GG^*)}^{\frac{1}{2}} \abs{P_E C},
\end{equation} 
where $E= \mathop{\rm Range}(G^*)$ is distributed uniformly on
$G_{N,n}$. Moreover, $\det{(GG^*)}^{1/2}$ and $\abs{P_E C}$ are
independent. The latter fact underlies the Gaussian representation of
intrinsic volumes, as proved by B. Tsirelson in \cite{Tsirelson} (see
also \cite{Vitale_Gaussian_rep}); it is also used in R. Vitale's
probabilistic derivation of the Steiner formula \cite{Vitale_Steiner}.
Passing between Gaussian vectors and random orthogonal projections is
useful in a variety of contexts, e.g., \cite{James}, \cite{Miles},
\cite{AS_projections}, \cite{Baryshnikov_Vitale},
\cite{BH_projections}, \cite{MTJ_projections}, \cite{DT_hypercubes},
\cite{PaoPiv_small}.  As we will show, however, it is a delicate
matter to use (\ref{eqn:splitting}) to prove limit theorems,
especially with the normalization required in Theorem
\ref{thm:CLTcube}. Our path will involve analyzing asymptotic
normality of $\abs{G B_{\infty}^N}$ before dealing with the quotient
$\abs{G B_{\infty}^N}/\det{(GG^*)}^{1/2}$.

The set
\begin{equation*}
GB_{\infty}^{N} = \left\{\sum_{i=1}^N \lambda_i g_i:\abs{\lambda_i}\ls
1, i=1,\ldots, N \right\}
\end{equation*}
is a random zonotope, i.e., a Minkowski sum of the random segments
$[-g_i,g_i]=\{\lambda g_i: \abs{\lambda}\ls 1 \}$.  By the well-known
zonotope volume formula (e.g. \cite{McMullen}),
$X_N=\abs{GB_{\infty}^N}$ satisfies
\begin{equation}
  \label{eqn:XNvolintro}
X_N = 2^n \sum_{1\ls i_1<\ldots<i_n\ls N}\abs{\det{[g_{i_1}\cdots
      g_{i_n}]}},
\end{equation} 
where $\det{[g_{i_1}\cdots g_{i_n}]}$ is the determinant of the matrix
with columns $g_{i_1},\ldots,g_{i_n}$. The quantity
\begin{equation*}
  U_N = \frac{1}{{N\choose n}} \sum_{1\ls i_1<\ldots<i_n\ls
    N}\abs{\det{[g_{i_1}\cdots g_{i_n}]}}
\end{equation*}
is a U-statistic and central limit theorems for U-statistics go back
to W. Hoeffding \cite{Hoeffding}.  In fact, formula
(\ref{eqn:XNvolintro}) for $X_N$ is simply a special case of
Minkowski's theorem on mixed volumes of convex sets (see
\S\ref{section:preliminaries}).  In \cite{Vitale_CLT}, R. Vitale
proved a central limit theorem for Minkowski sums of more general
random convex sets, using mixed volumes and U-statistics (discussed in
detail below). In particular, it follows from Vitale's results that
$X_N$ satisfies a central limit theorem, namely,
\begin{equation}
  \label{eqn:XNnormalVitale}
  \frac{X_N-\EE X_N}{s_{N,n}} \overset{d}{\rightarrow}
  \mathcal{N}(0,1),
\end{equation}
where $s_{N,n}$ is a certain conditional standard deviation (see
Theorem \ref{thm:Vitale_CLT}). Using Vitale's result and a more recent
randomization inequality for U-statistics \cite[Chapter 3]{dlPG}, we
show in \S\ref{section:proof} that $X_N$ satisfies a central limit
theorem with the canonical normalization:
\begin{equation}
  \label{eqn:XNnormal}
  \frac{X_N- \EE X_N}{\sqrt{\var{X_N}}} \overset{d}{\rightarrow}
  \mathcal{N}(0,1) \text{ as } N\rightarrow \infty.
\end{equation}

It is tempting to think that the latter central limit theorem for
$X_N$ easily yields Theorem \ref{thm:CLTcube}. However, for a family
of convex bodies $C=C_N\subset \R^N$, $N=n, n+1, \ldots$, asymptotic
normality of $\abs{GC}$ is not sufficient to conclude that
$\abs{P_{E(N)} C}$ is asymptotically normal. For example, if
$C=B_2^N$, then $\abs{GB_2^N}= \det{(GG^*)}^{1/2}\abs{B_2^n}$ is
asymptotically normal (e.g., \cite[Theorems 4.2.3, 7.5.3]{Anderson}),
however $\abs{P_{E(N)}B_2^N}$ is constant.

In fact, as we show in Proposition \ref{prop:ZNexpansion}, both $X_N$
and $\det{(GG^*)}^{1/2}$ contribute to asymptotic normality of
$Z_N=\abs{P_{E(N)}B_{\infty}^N}$, a technical difficulty that requires
careful analysis.  In particular, the aforementioned randomization
inequality from \cite[Chapter 3]{dlPG} is invoked again to deal with
the canonical normalization for $Z_N$ in Theorem \ref{thm:CLTcube}.
As a by-product, we also obtain the limiting behavior of the variance
of $Z_N$ as $N\rightarrow \infty$.

We mention that when $n=1$, Theorem \ref{thm:CLTcube}
implies that if $(\theta_N)$ is a sequence of random vectors with
$\theta_N$ distributed uniformly on the sphere $S^{N-1}$, then the
$\ell_1$-norm $\norm{\cdot}_1$ (the support function of the cube)
satisfies
\begin{equation*}
\frac{\norm{\theta_N}_1- \EE
  \norm{\theta_N}_1}{\sqrt{\var{\norm{\theta_N}_1}}}
\overset{d}{\rightarrow } \mathcal{N}(0,1) \text{ as } N\rightarrow
\infty.
\end{equation*}

The central limit theorem for $X_N$ in (\ref{eqn:XNnormal}) can be seen as a
counter-part to a recent result of I. B\'{a}r\'{a}ny and V. Vu
\cite{BaranyVu} for convex hulls of Gaussian vectors.  In particular,
when $n\gr 2$ the quantity $V_N = \abs{\conv{g_1,\ldots,g_N}}$
satisfies
\begin{equation*}
  \frac{V_N- \EE V_N}{\sqrt{\var{V_N}}}\overset{d}{\rightarrow}
  \mathcal{N}(0,1) \text{ as } N\rightarrow \infty;
\end{equation*} 
see the latter article for the corresponding Berry-Esseen type
estimate.  The latter result is one of several recent deep central
limit theorems in stochastic geometry concerning random convex hulls, e.g.,
\cite{Reitzner_CLT}, \cite{Vu_CLT}, \cite{BaranyReitzner}.  The
techniques used in this paper are different and the main focus here is
to understand the Grassmannian setting.

Lastly, for a thorough exposition of the properties of the cube, see
\cite{Zong}.

\section{Preliminaries}
\label{section:preliminaries}

The setting is $\R^n$ with the usual inner-product $\langle \cdot,
\cdot \rangle$ and Euclidean norm $\norm{\cdot}_2$; $n$-dimensional
Lebesgue measure is denoted by $\abs{\cdot}$. For sets $A,B \subset
\R^n$ and scalars $\alpha, \beta\in \R$, we define $\alpha A +\beta B$
by usual scalar multiplication and Minkowski addition: $\alpha A
+\beta B = \{\alpha a +\beta b: a\in A, b\in B\}$.

\subsection{Mixed volumes}

The mixed volume $V(K_1,\ldots,K_n)$ of compact convex sets
$K_1,\ldots,K_n$ in $\R^n$ is defined by
\begin{equation*}
   V(K_1,\ldots,K_n)  = \frac{1}{n!} \sum_{j=1}^n (-1)^{n+j} 
   \sum_{i_1<\ldots < i_j}\vol{K_{i_1}+\ldots+K_{i_j}}.
 \end{equation*}
By a theorem of Minkowski, if $t_1,\ldots,t_N$ are non-negative real
numbers then the volume of $K=t_1K_1+\ldots+t_NK_N$ can be expressed
as
\begin{equation}
  \label{eqn:Minkowski}
  \vol{K} = 
  \sum_{i_1=1}^N\cdots \sum_{i_n=1}^N 
  V(K_{i_1},\ldots,K_{i_n})t_{i_1}\cdots t_{i_n}.
\end{equation}  
The coefficients $V(K_{i_1},\ldots, K_{i_n})$ are non-negative and
invariant under permutations of their arguments.  When the $K_i$'s are
origin-symmetric line segments, say $K_i=[-x_i,x_i]=\{\lambda
x_i:\abs{\lambda}\ls 1\}$, for some $x_1,\ldots,x_n\in \R^n$, we 
simplify the notation and write 
\begin{equation}
  \label{eqn:mixedvol_notation}
  V(x_1,\ldots,x_n)=V([-x_1,x_1],\ldots,[-x_n,x_n]).
\end{equation}
We will make use of the following properties:
\begin{itemize}
\item[(i)] $V(K_1,\ldots,K_n)>0$ if and only if there are line
  segments $L_i\subset K_i$ with linearly independent directions.
\item[(ii)] If $x_1,\ldots, x_n\in \R^n$, then
  \begin{equation}
    \label{eqn:mixedvol_det}
    n!V(x_1,\ldots,x_n) = 2^n\abs{\det{[x_1\cdots
          x_n]}},
  \end{equation}
  where $\det{[x_1\cdots x_n]}$ denotes the determinant of the matrix
  with columns $x_1,\ldots,x_n$.
\item[(iii)]$V(K_1,\ldots,K_n)$ is increasing in each argument (with
  respect to inclusion).
\end{itemize}
For further background we refer the reader to \cite[Chapter
5]{Schneider} or \cite[Appendix A]{Gardner}.

A {\it zonotope} is a Minkowski sum of line segments. If
$x_1,\ldots,x_N$ are vectors in $\R^n$, then 
\begin{equation*}
  \sum_{i=1}^N[-x_i,x_i] 
  =\left\{\sum_{i=1}^N\lambda_i x_i: \abs{\lambda_i}\ls 1, \;i=1,\ldots,N\right\}.
\end{equation*} 
Alternatively, a zonotope can be seen as a linear image of the cube
$B_{\infty}^N = [-1,1]^N$.  If $x_1,\ldots,x_N\in \R^n$, one can view
the $n\times N$ matrix $X = [x_1\cdots x_N]$ as a linear operator from
$\R^N$ to $\R^n$; in this case, $X B_{\infty}^N = \sum_{i=1}^N
[-x_i,x_i]$.

By (\ref{eqn:Minkowski}) and properties (i) and (ii) of mixed volumes,
the volume of $\sum_{i=1}^N[-x_i,x_i]$ satisfies
\begin{equation}
  \label{eqn:zonotope_volformula}
\Bigl\lvert\sum_{i=1}^N[-x_i,x_i]\Bigr\rvert = 2^n \sum_{1\ls
  i_1<\ldots<i_n\ls N} \abs{\det{[x_{i_1}\cdots x_{i_n}]}}.
\end{equation}
Note that for $x_1,\ldots,x_n\in \R^n$, 
\begin{equation}
  \label{eqn:det_formula}
  \abs{\det{[x_1\cdots x_n]}} = \norm{x_1}_2\norm{P_{F_1^{\perp}}x_2}_2
  \cdots \norm{P_{F_{n-1}^{\perp}}x_n }_2,
\end{equation}
where $F_k=\mathop{\rm span}\{x_1,\ldots,x_k\}$ for $k=1,\ldots,n-1$
(which can be proved using Gram-Schmidt orthogonalization, e.g.,
\cite[Theorem 7.5.1]{Anderson}).

We will also use the Cauchy-Binet formula.  Let $x_1,\ldots,
x_N\in \R^n$ and let $X$ be the $n\times N$ matrix with columns
$x_1,\ldots,x_N$, i.e., $X= [x_1\cdots x_N]$.  Then
\begin{equation}
  \label{eqn:CauchyBinet}
  \det{(XX^*)}^{\frac{1}{2}} = \sum_{1\ls i_1<\ldots <i_n\ls N}
  \det{[x_{i_1}\cdots x_{i_n}]}^2;
\end{equation}for a proof, see, e.g., \cite[\S 3.2]{EvansGariepy}.

\subsection{Slutsky's theorem}

We will make frequent use of Slutsky's theorem on convergence of
random variables (see, e.g., \cite[\S 1.5.4]{Serfling}).

\begin{theorem}
  \label{thm:Slutsky}
  Let $(X_N)$ and $(\alpha_N)$ be sequences of random variables.
  Suppose that $X_N\overset{d}{\rightarrow} X_0$ and
  $\alpha_N\overset{\mathbb{P}}{\rightarrow} \alpha_0$, where
  $\alpha_0$ is a finite constant. Then
\begin{equation*}
  X_N + \alpha_N \overset{d}{\rightarrow} X_0 + \alpha_0
\end{equation*}and 
\begin{equation*}
  \alpha_N X_N \overset{d}{\rightarrow} \alpha_0 X_0.
\end{equation*}
\end{theorem}

Slutsky's theorem also applies when the $X_N$'s take values in $\R^k$
and satisfy $X_N\overset{d}{\rightarrow} X_0$ and $(A_N)$ is a
sequence of $m\times k$ random matrices such that
$A_N\overset{\mathbb{P}}{\rightarrow} A_0$ and the entries of $A_0$
are constants. In this case, $A_N X_N\overset{d}{\rightarrow} A_0
X_0$.

\section{U-statistics}
\label{section:Ustats}
In this section, we give the requisite results from the theory of
U-statistics needed to prove asymptotic normality of $X_N$ and $Z_N$
stated in the introduction.  For further background on U-statistics,
see e.g.  \cite{Serfling}, \cite{Rubin_Vitale}, \cite{dlPG}.

Let $X_1,X_2,\ldots$ be a sequence of i.i.d. random variables with
values in a measurable space $(S, \mathcal{S})$. Let $h:S^m
\rightarrow \R$ be a measurable function.  For $N\gr m$, the
U-statistic of order $m$ with kernel $h$ is defined by
\begin{equation}
  \label{eqn:Ustat}
  U_N = U_N(h) = \frac{(N-m)!}{N!} 
  \sum_{(i_1,\ldots,i_m)\in I_N^m} h(X_{i_1},\ldots,X_{i_m}), 
\end{equation}
where 
\begin{equation*}
  I_N^m=\left\{ (i_1,\ldots,i_m):i_j\in \N, 1\ls i_j\ls N, i_j\not =
  i_k \text{ if } j\not =k \right\}.
\end{equation*}
When $h$ is symmetric, i.e., $h(x_1, \ldots, x_m) = h(x_{\sigma(1)},
\ldots , x_{\sigma(m)})$ for every permutation $\sigma$ of $m$
elements, we can write
\begin{equation}
\label{eqn:Ustat_symmetric}
  U_N = U(X_1, \ldots ,X_N) = \frac{1}{ {N \choose m}} \sum_{1\ls
    i_1<\ldots<i_m\ls N} h( X_{i_1}, \ldots , X_{i_m});
\end{equation}
here the sum is taken over all ${N \choose m}$ subsets $\{i_1,\ldots,
i_m\}$ of $\{1,\ldots , N\}$.

Using the latter notation, we state several well-known results, due to
Hoeffding (see, e.g., \cite[Chapter 5]{Serfling}).

\begin{theorem} 
  \label{thm:Ustat_basics}
  For $N\gr m$, let $U_N$ be a statistic with kernel $h:S^m
  \rightarrow \R$.  Set $\zeta = \var{\EE [h(X_1,\ldots,X_m) | X_1]}$.
\begin{itemize}
\item[(1)]The variance of $U_N$ satisfies
\begin{equation*}
  \var{U_N} = \frac{m^2 \zeta }{N} + O(N^{-2}) \text{ as }N\rightarrow
  \infty.
\end{equation*}
\item[(2)] If $\EE \abs{h(X_1,\ldots,X_m)} <\infty$, then $U_N
  \overset{a.s.}{\rightarrow} \EE U_N$ as $N\rightarrow
  \infty$.
\item[(3)] If $\EE h^2(X_1, \ldots , X_m)< \infty$ and $\zeta >0$, then
  \begin{equation*}
    \sqrt{N} \left(\frac{ U_{N}- \EE U_N}{ m\sqrt{\zeta}
    }\right) \overset{d}{\rightarrow} {\cal{N}}(0,1) \text{ as }
    N\rightarrow \infty.
  \end{equation*}
\end{itemize}
\end{theorem}

The corresponding Berry-Esseen type bounds are also available (see,
e.g,. \cite[page 193]{Serfling}), stated here in terms of the function
\begin{equation*}
  \Phi(t)=\frac{1}{\sqrt{2\pi}}\int_{-\infty}^t e^{-s^2/2}ds.
\end{equation*}

\begin{theorem}
  With the preceding notation, suppose that $\xi=\EE \abs{h(X_{1},
    \ldots , X_{m})}^3 <\infty$ and
  \begin{equation*}
    \zeta = \var{\EE [ h(X_1,\ldots,X_m) | X_1]}>0.
\end{equation*}
Then
  \begin{equation*}
    \sup_{t\in \R} \left\lvert\Prob{ \sqrt{N}\left(\frac{U_N-\EE
        U_N}{m\sqrt{\zeta}} \right) \ls t}- \Phi(t)\right\rvert \ls
    \frac{c \xi}{ (m^{2} \zeta)^{\frac{3}{2}}\sqrt{N}},
  \end{equation*}
  where $c>0$ is an universal constant. 
\end{theorem}

\subsection{U-statistics and mixed volumes}

Let $\mathcal{C}_n$ denote the class of all compact, convex sets in
$\R^n$.  A topology on $\mathcal{C}_N$ is induced by the Hausdorff
metric
\begin{equation*}
   \delta^H(K,L) = \inf\{\delta >0: K \subset L +\delta B_2^n,
   L\subset K+\delta B_2^n\},
\end{equation*}
where $B_2^n$ is the Euclidean ball of radius one. A random convex set
is a Borel measurable map from a probability space into
$\mathcal{C}_n$. A key ingredient in our proof is the following
theorem for Minkowski sums of random convex sets due to R. Vitale
\cite{Vitale_CLT}; we include the proof for completeness.

\begin{theorem} 
  \label{thm:Vitale_CLT}
  Let $n\gr 1$ be an integer. Suppose that $K_1,K_2,\ldots$ are
  i.i.d. random convex sets in $\R^n$ such that $\EE \sup_{x\in
    K_1}\norm{x}_2<\infty$. Set $V_N =\abs{\sum_{i=1}^N K_i}$ and
  suppose that $\EE V(K_1,\ldots,K_n)^2<\infty$ and furthermore that
  $\zeta = \var{\EE [V(K_1,\ldots,K_n)| K_1]} >0$. Then
\begin{equation*}
  \sqrt{N}\left(\frac{V_N - \EE V_N}{ (N)_n n\sqrt{\zeta}}\right)
  \overset{d}{\rightarrow} \mathcal{N}(0,1) \text{ as } N\rightarrow
  \infty,
\end{equation*}where $(N)_n=\frac{N!}{(N-n)!}$.
\end{theorem}

\begin{proof}Taking $h:(\mathcal{C}_n)^n \rightarrow \R$ to be
  $h(K_1,\ldots,K_n) = V(K_1,\ldots,K_n)$ and using
  (\ref{eqn:Minkowski}), we have
  \begin{equation}
    \label{eqn:mixed_ustat}
    \frac{1}{(N)_n}V_N = U_N +
    \frac{1}{(N)_n}\sum_{(i_1,\ldots,i_n)\in
      J}V(K_{i_1},\ldots,K_{i_n})
  \end{equation}
  where
  \begin{equation*}
    U_N = \frac{1}{(N)_n} \sum_{(i_1,\ldots,i_n)\in I_N^n} V(K_{i_1},\ldots,K_{i_n}),
  \end{equation*} 
  and $J=\{1,\ldots,N\}^n\backslash I_N^n$. Note that $\abs{J}/(N)_n =
  O(\frac{1}{N})$ and thus the second term on the right-hand side of
  (\ref{eqn:mixed_ustat}) tends to zero in probability. Applying
  Theorem \ref{thm:Ustat_basics}(3) and Slutsky's theorem leads to the
  desired conclusion.
\end{proof}

In the special case when the $K_i$'s are line segments, say $K_i =
[-X_i, X_i]$ where $X_1,X_2,\ldots$ are i.i.d. random vectors in
$\R^n$, the assumptions in the latter theorem can be readily verified
by using (\ref{eqn:mixedvol_det}).  Furthermore, if the $X_i$'s are
rotationally-invariant, the assumptions simplify further as follows
(essentially from \cite{Vitale_CLT}, stated here in a form that best
serves our purpose).

\begin{corollary}
  \label{cor:ri_zonotopes}
  Let $X=R\theta$ be a random vector such that $\theta$ is uniformly
  distributed on the sphere $S^{n-1}$ and $R\gr 0$ is independent of
  $\theta$ and satisfies $\EE R^2 <\infty$ and $\var{R}>0$.  For each
  $i=1,2,\ldots$, let $X_i=R_i\theta_i$ be independent copies of $X$.
  Let $D_n=\abs{\det{[\theta_1\cdots \theta_n]}}$ and set
  \begin{equation*}
    \zeta_1 = 4^n  \var{R} \EE^{2(n-1)} R \EE^2 D_n.
  \end{equation*}
  Then
  $V_N=\abs{\sum_{i=1}^N [-X_i,X_i]}$ satisfies
  \begin{equation*}
    \sqrt{N}\left(\frac{V_N - \EE V_N}{{N\choose n} n\sqrt{\zeta_1}}\right)
    \rightarrow\mathcal{N}(0,1) \text{ as } N\rightarrow \infty.
  \end{equation*}
\end{corollary}

\begin{proof}
  Plugging $X_i=R_i\theta_i$, $i=1,\ldots,n$, into
  (\ref{eqn:mixedvol_det}) gives
  \begin{equation}
    \label{eqn:mixedvol_lengths}
    n! V(X_1,\ldots,X_n) = 2^n R_1\cdots R_n D_n.
   \end{equation} 
  By (\ref{eqn:det_formula}), 
  \begin{equation}
    \label{eqn:Dnexpansion}
    D_n = \norm{\theta_1}_2\norm{P_{{F_1}^{\perp}}\theta_2}_2\cdots
    \norm{P_{F_{n-1}^{\perp}}\theta_n}_2,
  \end{equation}
  with $F_k = \mathop{\rm span}\{\theta_1,\ldots,\theta_k\}$ for
  $k=1,\ldots,n-1$. In
  particular, $D_n\ls 1$ and thus (\ref{eqn:mixedvol_lengths}) implies
  \begin{equation*}
    \EE V(X_1,\ldots,X_n)^2 \ls \frac{4^n}{(n!)^2}\EE^n R^2 <\infty.
  \end{equation*}
  Using (\ref{eqn:mixedvol_lengths}) once more, together with
  (\ref{eqn:Dnexpansion}), we have
  \begin{equation}
    \label{eqn:cond_var}
    n!\EE[ V(X_1,\ldots,X_n)\vert X_1 ] = 2^n R_1 \EE R_2\cdots \EE
    R_n \EE D_n;
  \end{equation} 
  here we have used the fact that $\EE
  \norm{P_{{F_k}^{\perp}}\theta_{k+1}}_2$ depends only on the
  dimension of $F_k$ (which is equal to $k$ a.s.) and that
  $\norm{\theta_1}_2=1$ a.s.  By (\ref{eqn:cond_var}) and our assumption
  $\var{R}>0$, we can apply Theorem \ref{thm:Vitale_CLT} with
  \begin{equation*}
    \zeta = \var{\EE[V(X_1,\ldots,X_n)| X_1]} = \frac{\zeta_1}{(n!)^2} >0,
  \end{equation*}
  where $\zeta_1$ is defined in the statement of the corollary.
\end{proof}

For further information on Theorem \ref{thm:Vitale_CLT}, including a
CLT for the random sets themselves, or the case when $\zeta=0$, see
\cite{Vitale_CLT} or \cite[Pg 232]{Molchanov}; see also \cite{Vitale_Area}.

Corollary \ref{cor:ri_zonotopes} implies the first central limit
theorem for $X_N$ stated in the introduction
(\ref{eqn:XNnormalVitale}).  However, to recover the central limit
theorem for $X_N$ in (\ref{eqn:XNnormal}), involving the variance
$\var{X_N}$ and not a conditional variance, some additional tools are
needed.

\subsection{Randomization}

In this subsection, we discuss a randomization inequality for
U-statistics. It will be used for variance estimates, the proof of the
central limit theorem for $X_N$ in (\ref{eqn:XNnormal}) and it will
also play a crucial role in the proof of Theorem \ref{thm:CLTcube}.

Using the notation at the beginning of \S \ref{section:Ustats},
suppose that $h:(\R^n)^m\rightarrow \R$ satisfies $\EE
\abs{h(X_1,\ldots,X_m)} <\infty$ and let $1<r\ls m$.  Following
\cite[Definition 3.5.1]{dlPG}, we say that $h$ is degenerate of order
$r-1$ if
\begin{equation*}
  \EE_{X_r,\ldots,X_m}h(x_1,\ldots,x_{r-1},X_r,\ldots,X_m) = \EE
  h(X_1,\ldots, X_m)
\end{equation*}
for all $x_1,\ldots,x_{r-1}\in \R^n$, and the function
\begin{equation*}
  S^r \ni (x_1,\ldots,x_r)\mapsto
  \EE_{X_{r+1},\ldots,X_m}h(x_1,\ldots,x_r,X_{r+1},\ldots,X_m)
\end{equation*}
is non-constant.  If $h$ is not degenerate of any positive order $r$,
we say it is non-degenerate or degenerate of order $0$. We will make
use of the following randomization theorem, which is a special case of
\cite[Theorem 3.5.3]{dlPG}.  

\begin{theorem}
  \label{thm:randomization}
  Let $1\ls r \ls m$ and $p\gr 1$. Suppose that $h:S^m\rightarrow \R$
  is degenerate of order $r-1$ and $\EE
  \abs{h(X_1,\ldots,X_m)}^p<\infty$.  Set
  \begin{equation*}
    f(x_1,\ldots,x_m) = h(x_1,\ldots,x_m)-\EE h(X_1,\ldots,X_m).
  \end{equation*}
  Let $\eps_1,\ldots,\eps_N$ denote i.i.d. Rademacher random
  variables, independent of $X_1,\ldots,X_N$.  Then
  \begin{eqnarray*}
    \lefteqn{\EE\bigl\vert\sum_{(i_1,\ldots,i_m)\in I_N^m}
      f(X_{i_1},\ldots,X_{i_m}) \bigr\vert^p}\\ & & \simeq_{m,p}
    \EE\bigl\lvert\sum_{(i_1,\ldots,i_m)\in I_N^m} \eps_{i_1}\cdots \eps_{i_r}
    f(X_{i_1},\ldots, X_{i_m})\bigr\rvert^p.
  \end{eqnarray*} 
\end{theorem}
Here $A\simeq_{m,p}B$ means $C^{\prime}_{m,p} A \ls B \ls
C^{\prime\prime}_{m,p} A$, where $C^{\prime}_{m,p}$ and
$C^{\prime\prime}_{m,p}$ are constants that depend only on $m$ and
$p$.

\begin{corollary} 
  \label{cor:randomization}
  Let $\mu$ be probability measure on $\R^n$, absolutely continuous
  with respect to Lebesgue measure.  Suppose that $X_1,\ldots,X_N$ are
  i.i.d. random vectors distributed according to $\mu$.  Let $p\gr 2$
  and suppose $\EE \abs{\det{[X_1\cdots X_n]}}^p<\infty$. Define
  $f:(\R^n)^n \rightarrow \R$ by
\begin{equation*}
  f(x_1,\ldots,x_n) = \abs{\det{[x_1\cdots x_n]}} - \EE \abs{\det{
    [X_1\cdots X_n]}}.
\end{equation*}Then 
\begin{eqnarray*}
\EE \bigl\lvert \sum_{1\ls i_1<\ldots<i_n\ls N}
f(X_{i_1},\ldots,X_{i_n}) \bigr\rvert^p 
\ls C_{n,p} N^{p(n -\frac{1}{2})} \EE\abs{f(X_1,\ldots,X_n)}^p,
 \end{eqnarray*}
where $C_{n,p}$ is a constant that depends on $n$ and $p$.
\end{corollary}

\begin{proof}
  Since $\mu$ is absolutely continuous, $\dimm{\mathop{\rm
      span}\{X_1,\ldots, X_k\}}=k$ a.s. for $k=1,\ldots, n$. Moreover,
  $ f(ax_1,\ldots, x_n) = \abs{a}f(x_1,\ldots,x_n)$ for any $a\in \R$,
  hence $f$ is non-degenerate (cf. (\ref{eqn:det_formula})). Thus we
  may apply Theorem \ref{thm:randomization} with $r=1$:
\begin{eqnarray*}
  \EE \Bigl\lvert \sum_{1\ls i_1<\ldots<i_n\ls N} n!
  f(X_{i_1},\ldots,X_{i_n}) \Bigr\rvert^p &= &\EE \Bigl\lvert 
    \sum_{(i_1,\ldots,i_n)\in I_N^n} f(X_{i_1},\ldots,X_{i_n})
    \Bigr\rvert^p\\ & \ls & C_{n,p} \EE \Bigl\lvert
    \sum_{(i_1,\ldots,i_n)\in I_N^n} \eps_{i_1}
    f(X_{i_1},\ldots,X_{i_n}) \Bigr\rvert^p.
\end{eqnarray*}
Suppose now that $X_1,\ldots,X_N$ are fixed. Taking expectation in
${\bf \eps }= (\eps_1,\ldots,\eps_N)$ and appling Khintchine's
inequality and then H\"{o}lder's inequality twice, we have
\begin{eqnarray*}
  \lefteqn{\EE_{\bf \eps} \Bigl\lvert \sum_{(i_1,\ldots,i_n)\in I_N^n}
  \eps_{i_1} f(X_{i_1},\ldots,X_{i_n}) \Bigr\rvert^p}\\
  & = & \EE_{\bf \eps}\Bigl\lvert
  \sum_{i_1=1}^N \eps_{i_1}
  \sum_{\substack{(i_2,\ldots,i_n)\\(i_1,\ldots,i_n)\in I_N^n}}
  f(X_{i_1},\ldots,X_{i_n})\Bigr\rvert^p\\ & \ls & C\Bigl\lvert
  \sum_{i_1=1}^N \Bigl(
  \sum_{\substack{(i_2,\ldots,i_n)\\(i_1,\ldots,i_n)\in I_N^n}}
  f(X_{i_1},\ldots,X_{i_n})\Bigr)^2\Bigr\rvert^{\frac{p}{2}}\\ & \ls &
  C\left({N-1 \choose n-1} (n-1)!\right)^{\frac{p}{2}}
  \Bigl\lvert \sum_{(i_1,\ldots,i_n)\in
    I_N^n} f(X_{i_1},\ldots,X_{i_n})^2\Bigr\rvert^{\frac{p}{2}} \\ &
  \ls & C\left({N-1 \choose n-1} (n-1)!\right)^{\frac{p}{2}} \left({N
    \choose n }n!\right)^{\frac{p-2}{2}} \sum_{(i_1,\ldots,i_n)\in
    I_N^n} \abs{f(X_{i_1},\ldots,X_{i_n})}^p,
\end{eqnarray*}
where $C$ is an absolute constant. Taking expectation in the $X_i$'s
gives
\begin{eqnarray*}
  \lefteqn{\EE \Bigl\lvert \sum_{(i_1,\ldots,i_n)\in I_N^n} \eps_{i_1}
    f(X_{i_1},\ldots,X_{i_n}) \Bigr\rvert^p}\\ & \ls & \left({N-1
    \choose n-1} (n-1)!\right)^{\frac{p}{2}} \left({N \choose n
  }n!\right)^{\frac{p-2}{2}} {N\choose n}n! \EE
  \abs{f(X_{1},\ldots,X_{n})}^p. 
\end{eqnarray*}
The proposition follows as stated by using the estimate ${N\choose
  n}\ls (eN/n)^n$.
\end{proof}

\section{Proof of Theorem \ref{thm:CLTcube}}
\label{section:proof}

As explained in the introduction, our first step is identity
(\ref{eqn:splitting}), the proof of which is included for
completeness.

\begin{proposition}
  Let $N\gr n$ and let $G$ be an $n\times N$ random matrix with
  i.i.d. standard Gaussian entries.  Let $C\subset \R^N$ be a convex
  body.  Then
  \begin{equation}
    \label{eqn:splitting_prop}
    \vol{GC} = \det{(GG^*)}^{\frac{1}{2}} \vol{P_E C},
  \end{equation}
  where $E = \mathop{\rm Range(G^*)}$. Moreover, $E$ is distributed
  uniformly on  $G_{N,n}$ and $\det{(GG^*)}^{\frac{1}{2}}$ and $\vol{P_E C}$
  are independent.
\end{proposition}

\begin{proof}
  Identity (\ref{eqn:splitting_prop}) follows from polar
  decomposition; see, e.g., \cite[Theorem 2.1(iii)]{PaoPiv_small}.  To
  prove that the two factors are independent, we note that if $U$ is
  an orthogonal transformation, we have $\mathop{\rm det}(GG^*)^{1/2}
  = \mathop{\rm det}((GU)(GU)^*)^{1/2}$; moreover, $G$ and $GU$ have
  the same distribution.  Thus if $U$ is a random orthogonal
  transformation distributed according to the Haar measure, we have
  for $s, t\gr 0$,
  \begin{eqnarray*}
    \lefteqn{\mathbb{P}_{\otimes \gamma_n}\left(\mathop{\rm det}(GG^*)^{1/2}\ls
    s, \abs{P_{\mathop{\rm Range}(G^*)}C}\ls t\right)}\\
    & & = \mathbb{P}_{\otimes \gamma_n}\otimes
    \mathbb{P}_{U}\left(\mathop{\rm det}(GG^*)^{1/2}\ls s,
    \abs{P_{\mathop{\rm Range}(U^*G^*)}C}\ls t\right) \\ &  & =\EE_{\otimes
      \gamma_n}\left(\mathds{1}_{\{\mathop{\rm det}(GG^*)^{1/2}\ls s\}}
      \EE_U \mathds{1}_{\{\abs{P_{U^* \mathop{\rm Range}(G^*)} C}\ls t\}}\right) 
       \\ &  & =  \mathbb{P}_{\otimes\gamma_n}\left(\mathop{\rm
        det}(GG^*)^{1/2}\ls s\right)\nu_{N,n}\left(E\in
      G_{N,n}:\abs{P_{E}C}\ls t\right).
  \end{eqnarray*}
\end{proof}

Taking $C=B_{\infty}^N$ in (\ref{eqn:splitting_prop}), we set
\begin{equation}
  \label{eqn:XNdef}
  X_N=\vol{GB_{\infty}^N} = 2^n \sum_{1\ls i_1<\ldots<i_n\ls
    N}\abs{\det{[g_{i_1}\cdots g_{i_n}]}}
\end{equation} (cf. (\ref{eqn:zonotope_volformula})),
\begin{equation}
  \label{eqn:YNdef}
  Y_N = \det{(GG^*)}^{\frac{1}{2}} = \left(\sum_{1\ls i_1<\ldots < i_n
    \ls N} \det{[g_{i_1}\cdots g_{i_m}]}^2\right)^{\frac{1}{2}}
\end{equation} (cf. (\ref{eqn:CauchyBinet})), and 
\begin{equation}
  \label{eqn:ZNdef}
  Z_N = \vol{P_E B_{\infty}^N},
\end{equation}
where $E$ is distributed according to $\nu_{N,n}$ on $G_{N,n}$. Then
$X_N = Y_N Z_N$, where $Y_N$ and $Z_N$ are independent.  In order to
prove Theorem \ref{thm:CLTcube}, we start with several properties of
$X_N$ and $Y_N$.

\begin{proposition}
  \label{prop:XN} Let $X_N$ be as defined in (\ref{eqn:XNdef}).
  \begin{itemize}
  \item[(1)] For each $p\gr 2$,
    \begin{equation*}
      \EE\abs{X_N-\EE X_N}^p\ls C_{n,p} N^{p(n-\frac{1}{2})}.
    \end{equation*} 
  \item[(2)] The variance of $X_N$ satisfies
    \begin{equation*}
      \frac{\var{X_N}}{N^{2n-1}} \rightarrow c_n \text{ as }N\rightarrow \infty,
    \end{equation*}where $c_n$ is a positive constant that depends only on $n$.
  \item[(3)] $X_N$ is asymptotically normal; i.e.,
    \begin{equation*}
      \frac{X_N-\EE X_N}{\sqrt{\var{X_N}}} \overset{d}{\rightarrow}
      \mathcal{N}(0,1) \text{ as } N\rightarrow \infty.
    \end{equation*}
  \end{itemize}
\end{proposition}
  
\begin{proof}
  Statement (1) follows from Corollary \ref{cor:randomization}.  To
  prove (2), let $g$ be a random vector distributed according to
  $\gamma_n$. Then Corollary \ref{cor:ri_zonotopes} with $\zeta_1
  =4^n\var{\norm{g}_2}\EE^{2(n-1)}\norm{g}_2\EE^2 D_n $ yields
  \begin{equation}
    \label{eqn:XN_intermediate}
    \sqrt{N}\left(\frac{X_N- \EE X_N}{{N\choose
        n}n\sqrt{\zeta_1}}\right)
    \overset{d}{\rightarrow}\mathcal{N}(0,1) \text{ as }N\rightarrow
    \infty.
  \end{equation}
  On the other hand, by part (1) we have
  \begin{equation*}
    \frac{\EE \abs{X_N-\EE X_N}^4}{N^{4n-2}}\ls C_{n,p}.
  \end{equation*}
  This implies that the sequence $(X_N-\EE X_N)/N^{n-\frac{1}{2}}$ is uniformly integrable, hence
  \begin{equation*}
    \frac{\sqrt{\var{X_N}}}{N^{-\frac{1}{2}}{N\choose
        n}n\sqrt{\zeta_1}} \rightarrow 1 \text{ as }N\rightarrow
    \infty.
    \end{equation*}
  Part (3) now follows from (\ref{eqn:XN_intermediate}) and Slutsky's
  theorem.
\end{proof}

We now turn to $Y_N=\det{(GG^*)}^{\frac{1}{2}}$.  It is well-known
that
\begin{equation}
  \label{eqn:YNfactor}
  Y_N = \chi_N\chi_{N-1}\cdot\ldots\cdot \chi_{N-n+1},
\end{equation}
where $\chi_k=\sqrt{\chi_k^2}$ and the $\chi_k^2$'s are independent
chi-squared random variables with $k$ degrees of freedom, $k=N,
\ldots,N-n+1$ (see, e.g., \cite[Chapter 7]{Anderson}).  Consequently,
\begin{equation*}
  \EE Y_{N}^2 =\frac{N!}{(N-n)!} = N^n\left(1-\frac{1}{N}\right)
  \cdots\left(1-\frac{n-1}{N}\right).
\end{equation*} 
Additionally, we will use the following basic properties of $Y_N$.

\begin{proposition}
  \label{prop:YN}
  Let $Y_N$ be as defined in (\ref{eqn:YNdef}).
  \begin{itemize}
    \item[(1)] For each $p\gr 2$,
      \begin{equation*}
        \EE \abs{Y_N^2 - \EE Y_N^2}^p \ls C_{n,p} N^{p(n-\frac{1}{2})}.
      \end{equation*}
    \item[(2)]  The variance of $Y_N$ satisfies 
      \begin{equation*}
        \frac{\var{Y_{N}}}{N^{n-1}} \rightarrow \frac{n}{2} \text{ as }
        N\rightarrow \infty.
      \end{equation*}
      \item[(3)]  $Y_N^2$ is asymptotically normal; i.e.,
        \begin{equation*}
          \sqrt{N}\left(\frac{Y_N^2}{N^n}-1\right)
          \overset{d}{\rightarrow} \mathcal{N}(0,2n) \text{ as }
          N\rightarrow \infty.
        \end{equation*}
  \end{itemize}
\end{proposition}

\begin{proof}
  To prove part (1), we apply Corollary \ref{cor:randomization} to $Y_N^2$.

  To prove part (2), we use (\ref{eqn:YNfactor}) and define $Y_{N,n}$ by
  $Y_{N,n}=Y_N = \chi_N\chi_{N-1}\cdot\ldots\cdot \chi_{N-n+1}$ and
  procede by induction on $n$.  Suppose first that $n=1$ so that
  $Y_{N,1} =\chi_N$.  By the concentration of Gaussian measure (e.g.,
  \cite[Remark 4.8]{Pisier}), there is an absolute constant $c_1$
  such that $\EE\abs{\chi_N - \EE\chi_N}^4<c_1$ for all $N$, which implies that
  the sequence $(\chi_N- \EE \chi_N)_N$ is uniformly integrable.  By
  the law of large numbers $\chi_N/\sqrt{N}\rightarrow 1$ a.s. and
  hence $\EE\chi_N /\sqrt{N}\rightarrow 1$, by uniform integrability.
  Note that
  \begin{eqnarray*}
    \chi_N - \EE \chi_N &= & \frac{\chi_N^2 - \EE^2 \chi_N}{\chi_N + \EE
      \chi_N} \\ & = & \frac{\sqrt{N}}{\chi_N +\EE \chi_N}
    \frac{\chi_N^2-N}{\sqrt{N}} +\frac{\sqrt{N}}{\chi_N+\EE
      \chi_N}\frac{N-\EE^2\chi_N}{\sqrt{N}}.
  \end{eqnarray*}
  By Slutsky's theorem and the classical central limit theorem,
  \begin{equation*}
  \frac{\sqrt{N}}{\chi_N +\EE \chi_N} \frac{\chi_N^2-N}{\sqrt{N}}
  \overset{d}{\rightarrow}\frac{1}{2}\mathcal{N}(0,2) \text{ as } N\rightarrow \infty, 
  \end{equation*}while
  \begin{equation*}
    \frac{\sqrt{N}}{\chi_N+\EE \chi_N}\frac{N-\EE^2\chi_N}{\sqrt{N}}
    \rightarrow 0\text{ (a.s.)} \text{ as } N\rightarrow \infty,
  \end{equation*}since $\var{\chi_N}=N-\EE^2 \chi_N <c_1^{1/2}$.
  Thus
  \begin{equation*}
    \chi_N - \EE \chi_N \overset{d}{\rightarrow}
    \frac{1}{2}\mathcal{N}(0,2) = \mathcal{N}(0,\frac{1}{2}) \text{ as
    } N\rightarrow \infty.
  \end{equation*}
  Appealing again to uniform integrability of  $(\chi_N-\EE\chi_N)_N$, we have
  \begin{equation*}
    \var{Y_{N,1}}= \EE\abs{\chi_N- \EE\chi_N}^2 \rightarrow
    \frac{1}{2} \text{ as } N\rightarrow \infty.
  \end{equation*}
  
  Assume now that 
  \begin{equation*}
    \frac{\var{Y_{N-1,n-1}}}{N^{n-2}} \rightarrow \frac{n-1}{2} \text{
      as } N\rightarrow \infty.
  \end{equation*}
  Note that
  \begin{eqnarray*}
    \var{Y_{N,n}} & = & \EE \chi_N^2 \EE Y_{N-1,n-1}^2 - 
    \EE^2 \chi_N \EE^2 Y_{N-1,n-1} \\
    & = & \EE(\chi_N^2 - \EE^2 \chi_N) \EE Y_{N-1,n-1}^2 +
    \EE^2 \chi_N (\EE Y_{N-1,n-1}^2 - \EE^2 Y_{N-1,n-1}) \\
    & = & \var{\chi_N} \EE Y_{N-1,n-1}^2+ \EE^2 \chi_N \var{Y_{N-1,n-1}}.
  \end{eqnarray*}
  We conclude the proof of part (2) with
  \begin{equation*}
    \frac{\var{\chi_N}\EE Y_{N-1,n-1}^2}{N^{n-1}} \rightarrow \frac{1}{2},
  \end{equation*} and, using the inductive hypothesis, 
  \begin{equation*}
    \frac{\EE^2 \chi_N\var{Y_{N-1,n-1}}}{N^{n-1}}\rightarrow \frac{n-1}{2}.
  \end{equation*}

  Lastly, statement (3) is well-known (see, e.g., \cite[\S
    7.5.3]{Anderson}).
\end{proof}

The next proposition is the key identity for $Z_N$.  To state it we
will use the following notation:
\begin{equation}
  \label{eqn:Gaussian_determinant}
  \Delta_{n,p}^p= \EE \abs{\mathop{\rm det}[g_1\cdots g_n]}^p.
\end{equation}
Explicit formulas for $\Delta_{n,p}^p$ are well-known and follow from
identity (\ref{eqn:det_formula}); see, e.g., \cite[pg 269]{Anderson}.
\begin{proposition}
  \label{prop:ZNexpansion}
  Let $X_N, Y_N$ and $Z_N$ be as above (cf. (\ref{eqn:XNdef}) -
  (\ref{eqn:ZNdef})).  Then
  \begin{equation}
    \label{eqn:ZNexpansion}
    \frac{Z_N-\EE Z_N}{N^{\frac{n-1}{2}}} = \alpha_{N,n} \frac{X_N-\EE
      X_N}{N^{n-\frac{1}{2}}} -\beta_{N,n}\frac{Y_N^2-\EE
      Y_N^2}{N^{n-\frac{1}{2}}} - \delta_{N,n},
  \end{equation}
  where 
  \begin{itemize}
  \item[(i)] $\alpha_{N,n} \overset{a.s.}{\rightarrow} 1$ as
    $N\rightarrow \infty$;
  \item[(ii)] $\beta_{N,n} \overset{a.s.}{\rightarrow}
    \beta_n=\frac{2^{n-1}\Delta_{n,1}}{\Delta_{n,2}^2}$ as $N\rightarrow
    \infty$;
  \item[(iii)]$\delta_{N,n}\overset{a.s.}{\rightarrow} 0$ as
    $N\rightarrow \infty$.
  \end{itemize}
  Moreover, for all $p\gr 1$,
  \begin{equation*}
    \sup_{N\gr n+4p-1} 
    \max(\EE \abs{\alpha_{N,n}}^p, \EE \abs{\beta_{N,n}}^p, \EE \abs{\delta_{N,n}}^p)
    \ls C_{n,p}.
  \end{equation*}
 \end{proposition}

The latter proposition is the first step in passing from the quotient
$Z_N=X_N/Y_N$ to the normalization required in Theorem
\ref{thm:CLTcube}.  The fact that $N^{n-\frac{1}{2}}$ appears in both
of the denominators on the right-hand side of (\ref{eqn:ZNexpansion})
indicates that both $X_N$ and $Y_N^2$ must be accounted for in order
to capture the asymptotic normality of $Z_N$.

\begin{proof}
  Write
\begin{eqnarray*}
  Z_N - \EE Z_N & = & \frac{X_N}{Y_N} - \frac{\EE X_N}{\EE Y_N} \\
  & = &\frac{X_N - \EE X_N}{Y_N} -\left(\frac{\EE X_N}{\EE Y_N}-
    \frac{\EE X_N}{Y_N}\right) \\ 
  & = & \frac{X_N - \EE X_N}{Y_N} - \frac{(Y_N^2 - \EE Y_N^2+\var{Y_N})\EE X_N}
  {Y_N(Y_N+\EE Y_N) \EE Y_N } \\
  & = & \frac{X_N - \EE X_N}{Y_N} - \frac{(Y_N^2 - \EE Y_N^2)\EE X_N}
  {Y_N (Y_N+\EE Y_N)\EE Y_N } - \frac{\var{Y_N}\EE X_N}
  {Y_N(Y_N+\EE Y_N) \EE Y_N }.
\end{eqnarray*}
Thus
\begin{eqnarray*}
  \frac{Z_N-\EE Z_N}{N^{\frac{n-1}{2}}}
  =  \alpha_{N,n}
  \left(\frac{X_N-\EE X_N}{N^{n-\frac{1}{2}}}\right)
  -\beta_{N,n}
  \left(\frac{Y_N^2-\EE Y_N^2}{N^{n-\frac{1}{2}}}\right)
  -\delta_{N,n},
\end{eqnarray*}
which shows that (\ref{eqn:ZNexpansion}) holds with
\begin{equation*}
\alpha_{N,n} = \frac{N^{\frac{n}{2}}}{Y_N}, \;\;
\beta_{N,n} = \frac{N^{\frac{n}{2}}\EE X_N}{Y_N(Y_N+\EE Y_N)\EE Y_N}, \;\;
{\delta_{N,n}} = \beta_{N,n}\frac{\var{Y_N}}{N^{n-\frac{1}{2}}}.
\end{equation*}
Using the factorization of $Y_N$ in (\ref{eqn:YNfactor}) and applying
the SLLN for each $\chi_k$ ($k=N,\ldots,N-n+1$),  we have
\begin{equation*}
  \frac{Y_N}{\sqrt{\frac{N!}{(N-n)!}}} \overset{a.s.}{\rightarrow} 1 
  \text{ as } N\rightarrow \infty, 
\end{equation*} and hence
\begin{equation*} {\alpha_{N,n}}=\frac{N^{n/2}}{Y_N}
  \overset{a.s.}{\rightarrow} 1 \text{ as } N\rightarrow \infty.
\end{equation*}
By the Cauchy-Binet forumula (\ref{eqn:CauchyBinet}) and the SLLN for
U-statistics (Theorem \ref{thm:Ustat_basics}(2)), we have
\begin{equation*}
  \frac{1}{{N \choose n}}Y_N^2 \overset{a.s.}{\rightarrow} \Delta_{n,2}^2 \text{ as } 
  N\rightarrow \infty.
\end{equation*}
Thus \begin{equation*} \beta_{N,n}= \frac{2^n{N\choose n}\Delta_{n,1}
  }{Y_N^2 (1+\frac{\EE Y_N}{Y_N})}\frac{N^{n/2}}{\EE Y_N
  }\overset{a.s.}{\rightarrow} \frac{2^{n}\Delta_{n,1}}{2\Delta_{n,2}^2} \text{
    as } N\rightarrow \infty.
\end{equation*} 
By Proposition \ref{prop:YN}(2) and Slutsky's theorem, we also have
$\delta_{N,n}\overset{a.s.}{\rightarrow} 0 \text{ as } N\rightarrow
\infty$.  To prove the last assertion, we note that for $1\ls p\ls
(N-n+1)/2$,
\begin{equation*}
  \EE \left(\frac{N^{\frac{n}{2}}}{Y_N}\right)^p \ls C_{n,p},
\end{equation*}
where $C_{n,p}$ is a constant that depends on $n$ and $p$ only (see,
e.g., \cite[Lemma 4.2]{PaoPiv_small}).
\end{proof}

\begin{proof}[Proof of Theorem \ref{thm:CLTcube}]
 To simplify the notation, for $I = \{i_1,\ldots,i_n\}\subset
 \{1,\ldots,N\}$, write $d_I=\abs{\mathop{\rm det}[g_{i_1}\cdots
     g_{i_n}]}$. Applying Proposition \ref{prop:ZNexpansion}, we can write
\begin{equation*}
  \frac{Z_N-\EE Z_N}{N^{\frac{n-1}{2}}} 
  =\frac{{N\choose n}}{N^{n-\frac{1}{2}}}(U_N -\EE U_N)
  + A_{N,n} - B_{N,n}  - \delta_{N,n},
\end{equation*} where
\begin{equation*}
  U_N = \frac{1}{{N\choose n}} \sum_{\abs{I}=n} (2^n d_I -\beta_n d_I^2),
\end{equation*}
\begin{equation*}
  A_{N,n}=(\alpha_{N,n}-1)\left(\frac{X_N-\EE X_N}{N^{n-\frac{1}{2}}}\right),
\end{equation*}and
\begin{equation*}
B_{N,n}=(\beta_{N,n}-\beta_n)\left(\frac{Y_N^2-\EE Y_N^2}{N^{n-\frac{1}{2}}}\right).
\end{equation*}
Set $I_0=\{1,\ldots,n\}$. Applying Theorem \ref{thm:Ustat_basics}(3)
with
\begin{equation}
  \label{eqn:zeta_ZN}
  \zeta  = \var{\EE[ (2^n d_{I_0}-\beta_n d^2_{I_0})| g_1]},
\end{equation}
yields
\begin{equation*}
  \sqrt{N}\left(\frac{U_N - \EE U_N}{n\sqrt{\zeta}}\right)
  \overset{d}{\rightarrow} \mathcal{N}(0,1) \text{ as } N\rightarrow
  \infty.
\end{equation*} 

By Proposition \ref{prop:ZNexpansion},
$\alpha_{N,n}\overset{a.s.}{\rightarrow} 1$,
$\beta_{N,n}\overset{a.s.}{\rightarrow} \beta_n$ and
$\delta_{N,n}\overset{a.s.}{\rightarrow} 0$; moreover, each of the
latter sequences is uniformly integrable. Thus by H\"{o}lder's
inequality and Proposition \ref{prop:XN}(1)
\begin{equation*}
  \EE\abs{A_{N,n}}\leq (\EE\abs{\alpha_{N,n}-1}^2)^{1/2}C_{n}\rightarrow 0
  \text{ as } N\rightarrow \infty.
\end{equation*}Similarly, using Proposition \ref{prop:YN}(1), 
\begin{equation*}
  \EE\abs{B_{N,n}}\leq (\EE\abs{\beta_{N,n}-\beta_n}^2)^{1/2}C_{n}\rightarrow 0
  \text{ as } N\rightarrow \infty.
\end{equation*}
By Slutsky's theorem and the fact that ${N\choose n}/N^n\rightarrow
1/n!$ as $N\rightarrow \infty$, we have
\begin{equation}
  \label{eqn:ZN_int_normal}
  \frac{n!(Z_N - \EE Z_N)}{N^{\frac{n-1}{2}}n\sqrt{\zeta}}
  \overset{d}{\rightarrow} \mathcal{N}(0,1) \text{ as } N\rightarrow
  \infty.
\end{equation}
To conclude the proof of the theorem, it is sufficient to show that
\begin{equation}
  \label{eqn:ZNvar_ratio}
\frac{n!\sqrt{\var{Z_N}}}{N^{\frac{n-1}{2}}n \sqrt{\zeta}}
\rightarrow 1 \text{ as } N\rightarrow \infty.
\end{equation}
Once again we appeal to uniform integrability: by Proposition
\ref{prop:ZNexpansion},
\begin{equation*}
\frac{\abs{Z_N-\EE Z_N}}{N^{\frac{n-1}{2}}} \ls 2^n\abs{\alpha_{N,n}}
\frac{\abs{X_N-\EE X_N}}{N^{n-\frac{1}{2}}} + 
\abs{\beta_{N,n}}\frac{\abs{Y_N^2 -\EE Y_N^2}}{N^{n-\frac{1}{2}}}+
\abs{\delta_{N,n}}.
\end{equation*} 
By H\"{o}lder's inequality and Propositions \ref{prop:XN}(1),
\ref{prop:YN}(1) and \ref{prop:ZNexpansion}, 
\begin{equation*}
\sup_{N\gr n+8p-1} \left\lvert\frac{Z_N-\EE
  Z_N}{N^{\frac{n-1}{2}}}\right\rvert^p \ls C_{n,p},
\end{equation*} 
which, combined with (\ref{eqn:ZN_int_normal}), implies
(\ref{eqn:ZNvar_ratio}).
\end{proof}

\section*{Acknowledgements}

It is our pleasure to thank R. Vitale for helpful comments on an
earlier version of this paper.

\bibliographystyle{amsplain} 
\bibliography{CLTcubeBIB}
\vspace{2cm}
\pagebreak
\noindent\begin{minipage}[l]{\linewidth}
  Grigoris Paouris: {\tt grigoris@math.tamu.edu}\\
  Department of Mathematics, Texas A\&M University \\
  College Station, TX, 77843-3368\\
\end{minipage}

\noindent\begin{minipage}[l]{\linewidth}
  Peter Pivovarov: {\tt pivovarovp@missouri.edu}\\
  Mathematics Department, University of Missouri\\
  Columbia, MO, 65211\\
\end{minipage}

\noindent\begin{minipage}[l]{\linewidth}
  Joel Zinn: {\tt jzinn@math.tamu.edu}\\
  Department of Mathematics, Texas A\&M University\\
  College Station, TX, 77843-3368\\
\end{minipage}

\end{document}